\documentclass[reqno,12pt]{amsart}

\usepackage[a4paper,margin=1in]{geometry}
\usepackage[T1]{fontenc}
\usepackage{lmodern}
\usepackage{microtype}
\usepackage{amsmath,amssymb,amsthm,mathtools}
\usepackage{enumitem}
\usepackage{xcolor}
\usepackage{hyperref}

\definecolor{navy}{HTML}{173B57}
\definecolor{linkblue}{HTML}{245B7A}
\hypersetup{
  colorlinks=true,
  linkcolor=linkblue,
  citecolor=linkblue,
  urlcolor=linkblue,
  pdftitle={A Chain-Level Borsuk--Ulam Proof of Norine's Antipodal-Coloring Conjecture},
  pdfsubject={A chain-level Borsuk--Ulam proof of the monochromatic antipodal path theorem}
}

\newtheorem{theorem}{Theorem}[section]
\newtheorem{lemma}[theorem]{Lemma}
\newtheorem{proposition}[theorem]{Proposition}
\newtheorem{corollary}[theorem]{Corollary}

\theoremstyle{definition}
\newtheorem{definition}[theorem]{Definition}

\DeclareMathOperator{\conv}{conv}
\DeclareMathOperator{\pos}{pos}
\DeclareMathOperator{\relint}{relint}
\DeclareMathOperator{\rank}{rank}
\newcommand{\F}{\mathbb F_2}
\newcommand{\id}{\mathrm{id}}
\newcommand{\cell}{\mathrm{cell}}

\begin{document}

\title[]{A Chain-Level Borsuk--Ulam Obstruction Proof of Norine's Antipodal-Coloring Conjecture}

\author[]{Hehui Wu and Ningyuan Yang}

    \address[Hehui Wu]{Shanghai Center for Mathematical Sciences, Fudan University, Shanghai, China}
	\email{hhwu@fudan.edu.cn}

    \address[Ningyuan Yang]{School of Mathematical Sciences, Fudan University, Shanghai, China, and Extremal Combinatorics and Probability Group (ECOPRO), Institute for Basic Science (IBS), Daejeon, South Korea.}
	\email{nyyang23@m.fudan.edu.cn}

\begin{abstract}
We prove Norine's conjecture: every red--blue edge-coloring of the
\(n\)-dimensional hypercube \(Q_n\), \(n\geq2\), in which antipodal
edges have opposite colors contains a monochromatic path joining some
vertex to its antipode.  From a hypothetical counterexample we
construct an antipodally equivariant, augmentation-preserving chain
map from the cellular chains of the cubical boundary of a cube to
subdivision-invariant polyhedral chains on a sphere of one lower
dimension.  A purely algebraic chain-level Borsuk--Ulam obstruction
rules out this map.
\end{abstract}

\maketitle

\section{Introduction}\label{sec:introduction}

For \(n\geq0\), let \(Q_n\) be the graph on \(\{0,1\}^n\), with two
vertices adjacent when they differ in exactly one coordinate.  The map
\[
 A_n(x_1,\ldots,x_n)=(1-x_1,\ldots,1-x_n)
\]
is the antipodal involution of \(Q_n\).  Thus \(x\) and \(A_nx\) are
antipodal vertices, and an edge \(xy\) is antipodal to
\((A_nx)(A_ny)\).

A red--blue edge-coloring of \(Q_n\) is \emph{antipodal} if every pair
of antipodal edges receives opposite colors.  In 2008, Norine asked
whether every such coloring contains a monochromatic path between a
pair of antipodal vertices \cite{Norine}.  We prove that it does.

\begin{theorem}[Norine's conjecture]\label{thm:main}
For every \(n\geq2\), every antipodal red--blue edge-coloring of
\(Q_n\) contains a monochromatic path joining some vertex \(x\) to
its antipode \(A_nx\).
\end{theorem}

Theorem~\ref{thm:main} settles a problem that had remained open since
2008 despite substantial progress in low dimensions.  Feder and Subi
proved the stronger monochromatic geodesic statement for \(n\leq5\).
Dimension \(6\) was first verified computationally and later proved by
hand by West and Wise, while SAT methods subsequently reached
dimensions \(7\) and \(8\)
\cite{FederSubi,ZulkoskiEtAl,WestWise,FrankstonScheinerman,
KirchwegerEtAl}.  Thus, before the present work, all exact results were
confined to particular dimensions.  Theorem~\ref{thm:main} confirms
that the phenomenon persists throughout the entire sequence of
hypercubes.

Topological methods have a well-established place in combinatorics.  A
landmark is Lov\'asz's proof of Kneser's conjecture, which converts a
graph-coloring problem into a Borsuk--Ulam-type obstruction
\cite{Lovasz}; Borsuk--Ulam-type methods have since become a central
part of topological combinatorics \cite{Bjorner,Matousek}.  The
discrete cube problem has recently acquired a related topological
picture.  D\v{z}avoronok obtained a criterion for certain
antipodal colorings by working with centrally symmetric, simply
connected \(2\)-complexes \cite{Dzavoronok}.  In a continuous setting,
Ellis, Ivan, Leader, and Mackay recently proved a sharp theorem about
antipodal connectivity in covers of spheres \cite{EllisEtAl}.  If
\(S^m\) is covered by \(k\) open sets and \(m\geq2k-2\), then one of
the sets contains a path with antipodal endpoints; this can fail when
\(m<2k-2\).  This is a spherical covering analogue of the cube
problem, rather than a direct reformulation of the condition that
antipodal edges have opposite colors.  Nevertheless, both results
show that requiring a path, rather than merely a pair of antipodal
points, still leads to a dimension-sensitive topological obstruction.

Our proof makes this connection explicit for the discrete cube.  A
hypothetical counterexample determines a labeling that records the red
component containing each vertex and the red component containing its
antipode.  From this labeling we construct an antipodally equivariant,
augmentation-preserving chain map from the cellular chain complex of
the cubical boundary of a cube to the subdivision-invariant
polyhedral chain complex on a sphere of one lower dimension.  The
classical
Borsuk--Ulam theorem \cite{Borsuk,Matousek} rules out an antipodal
continuous map between such spheres.  Our construction need not arise
from a continuous map, so the contradiction is instead supplied by a
purely algebraic chain-level Borsuk--Ulam obstruction.  Thus the
obstruction is extracted directly from the monochromatic component
structure; no induction on the dimension or finite search is involved.

Several stronger versions remain open.  One may ask whether the
monochromatic path in Theorem~\ref{thm:main} can always be chosen to be
geodesic.  Feder and Subi proposed a different strengthening: in an
arbitrary red--blue coloring, with no antipodality assumption, there
should be an antipodal path with at most one color change
\cite{FederSubi}.  Requiring this path to be geodesic gives a third
conjecture.  The two geodesic conjectures are equivalent as statements
over all dimensions, although one implication passes to a cube of one
higher dimension \cite{LeaderLong,Soltesz}.

A quantitative relaxation asks how many color changes are always
sufficient in an arbitrary coloring.  Leader and Long asked whether
the answer is \(o(n)\) \cite{LeaderLong}.  After earlier linear bounds
\cite{Dvorak,KirchwegerEtAl}, Hollom proved that every red--blue
coloring has an antipodal geodesic with at most
\[
 \left(\frac{\pi}{2}+o(1)\right)\sqrt n
\]
color changes \cite{Hollom}.

\subsection{Proof strategy and organization}

The proof begins by reducing Theorem~\ref{thm:main} to the diagonal
rook theorem, which rules out an antipodally symmetric rook labeling
of the cube.  We associate a label \((a,b)\) with the root
\(e_a-e_b\) in
\[
 W_K=\left\{z\in\mathbb R^K:\sum_{i=1}^Kz_i=0\right\}.
\]
Swapping the two entries negates the root, so the antipodal rule for
the labeling becomes the usual antipodal symmetry on
\(S(W_K)\cong S^{K-2}\).

We pass from cubical cellular chains to simplicial chains by the
Freudenthal subdivision.  For each cubical face \(F\), let
\(\mathfrak F(F)\) be the mod--\(2\) sum of its maximal Freudenthal
simplices.  Internal simplicial facets occur twice and cancel, so
\(\mathfrak F\) is a chain map.  Moreover, the vertices of every
maximal simplex occurring in \(\mathfrak F(F)\) form a monotone cube
path, and hence their labels form a rook gallery.

For such a simplex with root labels \(v_0,\ldots,v_j\), we take the
spherical section of their positive cone when the roots are linearly
independent, and zero when they are dependent.  These assignments
define a linear map \(\rho\) to a polyhedral chain complex modulo
subdivision.  A mod--\(2\) degenerate radial cancellation identity and
the rook-gallery rank bound show that the composite
\[
 C_*^{\cell}(\partial_{\mathrm{top}}I^K;\F)
 \xrightarrow{\ \mathfrak F_*\ }
 C_*^{\mathrm{simp}}(\mathcal T_K;\F)
 \xrightarrow{\ \rho_*\ }
 C_*^{\mathrm{poly}}(S(W_K);\F)
\]
is an antipodally equivariant, augmentation-preserving chain map.
Here
\[
 \partial_{\mathrm{top}}I^K\cong S^{K-1},
 \qquad
 S(W_K)\cong S^{K-2}.
\]
The target has no chains above dimension \(K-2\), and its antipodal
norm operator has equal kernel and image in every degree.  The
chain-level Borsuk--Ulam obstruction therefore rules out the
composite.

Section~2 carries out the reduction to a rook labeling.  Section~3
establishes the algebraic chain-level Borsuk--Ulam obstruction.
Section~4 develops the polyhedral chains on the target sphere and
proves the required kernel--image identity for the antipodal norm
operator.  Section~5 combines cubical cellular chains, Freudenthal
gallery chains, and radial polyhedral chains to construct the
forbidden chain map and complete the proof.

\section{Reduction to a rook labeling}

For \(r\geq1\), write
\[
 [r]=\{1,\ldots,r\},\qquad
 \Omega_r=\{(a,b)\in[r]^2:a\ne b\},
 \qquad \tau(a,b)=(b,a).
\]

\begin{definition}
A \emph{rook labeling} of an \(N\)-dimensional cube with label set
\([r]\) is a map \(q:V(Q_N)\to\Omega_r\) such that
\[
 q(A_Nx)=\tau q(x)
\]
and, for every cube edge \(xy\), the labels \(q(x)\) and \(q(y)\)
have the same first coordinate or the same second coordinate.
\end{definition}

\begin{lemma}\label{lem:reduction}
Suppose that an antipodal red--blue coloring of \(Q_n\) contains no
monochromatic antipodal path, and let \(r\) be the number of connected
components of its red spanning subgraph.  For every
\[
 K\geq\max\{n,r\},
\]
there is a rook labeling
\[
 q:V(Q_K)\longrightarrow\Omega_K.
\]
\end{lemma}

\begin{proof}
Let \(G_{\mathrm{red}}\) be the spanning subgraph of \(Q_n\) whose edge
set consists of the red edges, and let \(C_1,\ldots,C_r\) be its
connected components; an isolated vertex is also a component.  Let
\(p(x)=i\) when \(x\in C_i\), and define
\[
 q_0(x)=\bigl(p(x),p(A_nx)\bigr).
\]
The two entries are distinct, since otherwise one red component would
contain \(x\) and \(A_nx\).  Also \(q_0(A_nx)=\tau q_0(x)\).

If \(xy\) is red, then \(p(x)=p(y)\).  If \(xy\) is blue, its antipodal
edge is red, so \(p(A_nx)=p(A_ny)\).  Thus \(q_0\) is a rook labeling.

Choose an injection \(\iota:[r]\hookrightarrow[K]\).  Using
\(Q_K=Q_n\times Q_{K-n}\), set
\[
 q(x,z)=\bigl(\iota(p(x)),\iota(p(A_nx))\bigr).
\]
Old-coordinate edges inherit the rook property from \(q_0\), while the
label is constant along every new-coordinate edge.  Moreover,
\[
 q(A_nx,A_{K-n}z)
 =\bigl(\iota(p(A_nx)),\iota(p(x))\bigr)
 =\tau q(x,z).
\]
Hence \(q\) is the required labeling.
\end{proof}

Theorem~\ref{thm:main} therefore reduces to the following statement.

\begin{theorem}\label{thm:diagonal-rook}
For every \(K\geq2\), there is no rook labeling
\(Q_K\to\Omega_K\).
\end{theorem}

\begin{proof}[Deduction of Theorem~\ref{thm:main}]
If a counterexample existed in \(Q_n\), let \(r\) be the number of
components of its red spanning subgraph and choose
\(K\geq\max\{n,r,2\}\).  Lemma~\ref{lem:reduction} would give a rook
labeling \(Q_K\to\Omega_K\), contrary to
Theorem~\ref{thm:diagonal-rook}.
\end{proof}

\section{An algebraic chain-level Borsuk--Ulam obstruction}

All chain complexes in this section are over \(\F\) and concentrated
in nonnegative degrees.  Thus a
\emph{chain complex} \(C_*\) consists of vector spaces \(C_i\),
\(i\geq0\), and linear maps
\[
 \partial_i^C:C_i\longrightarrow C_{i-1}
\]
such that
\[
 \partial_{i-1}^C\partial_i^C=0
 \qquad(i\geq2),
\]
with \(C_i=0\) for \(i<0\) and \(\partial_0^C=0\).  Elements of
\(C_i\) are called \(i\)-chains.  The identity
\(\partial^C\partial^C=0\) implies
\(\operatorname{im}\partial_{i+1}^C\subseteq\ker\partial_i^C\).
For \(i\geq1\), the complex is \emph{exact in degree \(i\)} when
these two subspaces are equal.

A \emph{chain map}
\[
 \phi:C_*\longrightarrow D_*
\]
is a collection of linear maps \(\phi_i:C_i\to D_i\) satisfying
\[
 \partial_i^D\phi_i=\phi_{i-1}\partial_i^C
 \qquad(i\geq1).
\]

An \emph{augmentation} of \(C_*\) is a linear map
\[
 \epsilon_C:C_0\longrightarrow\F
\]
such that \(\epsilon_C\partial_1^C=0\).  The pair
\((C_*,\epsilon_C)\) is then called an \emph{augmented chain
complex}.  Its augmented sequence begins
\[
 \cdots\longrightarrow C_1
 \xrightarrow{\ \partial_1^C\ }C_0
 \xrightarrow{\ \epsilon_C\ }\F
 \longrightarrow0.
\]
It is exact at \(C_0\) when
\(\ker\epsilon_C=\operatorname{im}\partial_1^C\), and exact at
\(\F\) when \(\operatorname{im}\epsilon_C=\F\).  A chain map
\(\phi:C_*\to D_*\) is \emph{augmentation-preserving} if
\[
 \epsilon_D\phi_0=\epsilon_C.
\]

A \emph{chain involution} on an augmented chain complex is a chain
map \(T:C_*\to C_*\) satisfying
\[
 T_i^2=\id_{C_i}\qquad(i\geq0),
 \qquad
 \epsilon_CT_0=\epsilon_C.
\]
If \(C_*\) and \(D_*\) have chain involutions \(T\) and \(U\), a
chain map \(\phi:C_*\to D_*\) is \emph{equivariant} if
\[
 U_i\phi_i=\phi_iT_i\qquad(i\geq0).
\]

\subsection{An equivariant chain obstruction}

The following algebraic obstruction is the form of the
Borsuk--Ulam principle needed below.  It applies directly to the
polyhedral target complex, without first realizing that complex as a
cellular chain complex.

\begin{lemma}\label{lem:dold-chain}
Let \((C_*,\epsilon_C,T)\) and \((D_*,\epsilon_D,U)\) be augmented
chain complexes with involutions over \(\F\), and let \(d\geq0\).
Assume
\[
 \operatorname{im}\epsilon_C=\F,\qquad
 \ker\epsilon_C=\operatorname{im}\partial_1^C,
\]
\[
 \ker\partial_i^C=\operatorname{im}\partial_{i+1}^C
 \qquad(1\leq i\leq d),
\]
\[
 D_i=0\quad(i>d),
 \qquad
 \ker(\id+U_i)=\operatorname{im}(\id+U_i)\quad(0\leq i\leq d).
\]
Then there is no equivariant augmentation-preserving chain map
\(C_*\to D_*\).
\end{lemma}

\begin{proof}
Write
\[
 \nu_C=\id+T,\qquad
 \nu_D=\id+U.
\]
These are degreewise maps. Since the coefficients lie in \(\F\),
\[
 \nu_C^2=\nu_D^2=0,\qquad
 \partial^C\nu_C=\nu_C\partial^C,\qquad
 \partial^D\nu_D=\nu_D\partial^D,\qquad
 \epsilon_C\nu_C=\epsilon_D\nu_D=0.
\]

Choose \(x_0\) with \(\epsilon_C(x_0)=1\). The source assumptions
give chains \(x_i\in C_i\), \(1\le i\le d+1\), satisfying
\[
 \epsilon_C(x_0)=1,\qquad
 \partial^C x_{i+1}=\nu_Cx_i
 \quad(0\le i\le d).
\]
Indeed,
\[
 \epsilon_C(\nu_Cx_0)=0,
\]
so exactness in degree \(0\) gives \(x_1\). Inductively, for
\(i\ge1\),
\[
 \partial^C(\nu_Cx_i)
 =\nu_C\partial^C x_i
 =\nu_C^2x_{i-1}
 =0.
\]
Exactness in degree \(i\) therefore supplies \(x_{i+1}\).

Suppose that \(\phi\) exists and put $y_i=\phi_i(x_i)$. Then
\[
 \partial^D y_{i+1}=\nu_Dy_i,
\]
and $y_{d+1}=0$ because \(D_{d+1}=0\).

Set \(z_{d+1}=0\). Descending from degree \(d\), choose
\(z_i\in D_i\), \(0\le i\le d\), such that
\[
 y_i=\partial^D z_{i+1}+\nu_Dz_i
 \qquad(0\le i\le d).
\]
For \(i=d\), this follows from
\[
 \nu_Dy_d=\partial^D y_{d+1}=0.
\]
If it holds in degree \(i+1\), then
\[
 \nu_D\bigl(y_i+\partial^Dz_{i+1}\bigr)
 =\partial^D\bigl(y_{i+1}+\nu_Dz_{i+1}\bigr)
 =(\partial^D)^2z_{i+2}
 =0,
\]
so the same kernel--image equality supplies \(z_i\). Finally,
\[
 \epsilon_D(y_0)
 =\epsilon_D(\partial^Dz_1)
 +\epsilon_D(\nu_Dz_0)
 =0,
\]
whereas augmentation preservation gives
\[
 \epsilon_D(y_0)
 =\epsilon_C(x_0)
 =1,
\]
a contradiction.
\end{proof}

If $U$ permutes a chosen basis of $D$, then
\[
\ker(\id+U)=\operatorname{im}(\id+U)
\]
if and only if the induced involution has no fixed basis element. Therefore, Lemma~\ref{lem:dold-chain} is a chain-level version of the \(\mathbb Z/2\mathbb Z\)-case of Dold's theorem \cite{Dold}. By the Equivariant Simplicial Approximation Theorem, it is a generalization of the Borsuk--Ulam obstruction to an antipodal map \(S^{d+1}\to S^d\).  We shall use it with the standard cellular chains of the cubical boundary \(\partial_{\mathrm{top}}I^K\) as the source. Their augmentation sends every vertex to \(1\), and their augmented complex is exact below the top dimension by cellular homology \cite[Section~2.2]{Hatcher}.

\section{The target sphere and polyhedral chains}

Fix \(K\geq2\) and put
\[
 W_K=\left\{z\in\mathbb R^K:\sum_{i=1}^Kz_i=0\right\},
 \qquad d=K-2.
\]
Since \(\dim W_K=K-1\), its unit sphere \(S(W_K)\) has dimension
\(d=K-2\), so
\(S(W_K)\cong S^d\).  Let
\(A_{W_K}:S(W_K)\to S(W_K)\) be the antipodal map \(x\mapsto-x\).
We now define a subdivision-invariant polyhedral chain complex on this
sphere and verify the algebraic antipodal property required in
Lemma~\ref{lem:dold-chain}.

\subsection{Spherical polytopes and affine models}

The rook construction produces spherical polytopes that need not meet
face-to-face.  We therefore build a polyhedral chain group in which a
polytope is identified with the sum of the pieces in any subdivision.

For \(j\geq0\), a \emph{spherical \(j\)-polytope} in \(S(W_K)\) is a set
\[
 P=S(W_K)\cap\mathcal C,\qquad \dim\mathcal C=j+1,
\]
where \(\mathcal C\subseteq W_K\) is a polyhedral cone satisfying
\[
 \mathcal C\cap(-\mathcal C)=\{0\}.
\]
By the Minkowski--Weyl theorem \cite[Chapter~1]{Ziegler},
\(\mathcal C\) admits a finite homogeneous half-space representation.
The nonempty faces of \(P\) are the intersections of the sphere with
the nonzero faces of \(\mathcal C\); write \(F\lessdot P\) when \(F\) is a
facet of \(P\).

A finite family of spherical \(j\)-polytopes
\(P_1,\ldots,P_s\) is a
\emph{face-to-face subdivision} of a spherical \(j\)-polytope
\(P=S(W_K)\cap\mathcal C\) if
\[
 P=\bigcup_{\alpha=1}^sP_\alpha,\qquad
 \relint P_\alpha\cap\relint P_\beta=\varnothing
 \quad(\alpha\ne\beta),
\]
and every nonempty \(P_\alpha\cap P_\beta\) is a common face, where
\(\relint\) is the relative interior taken in
\(S(W_K)\cap\operatorname{span}\mathcal C\).

Let $Q=\conv P$. Pointedness implies \(0\notin Q\): a convex representation of \(0\) by points of \(P\) would place some nonzero \(x\) and \(-x\) in \(\mathcal C\).  Let \(q\) be the
unique point of minimum norm in \(Q\), and define
\[
 \ell:W_K\longrightarrow\mathbb R,
 \qquad
 \ell(x)=\frac{\langle q,x\rangle}
                 {\|q\|^2},
 \qquad
 H_\ell=\{x\in W_K:\ell(x)=1\},
 \qquad
 P_\ell=\mathcal C\cap H_\ell.
\]
Then, $\ell$ is a positive function on $W_K$.

\begin{lemma}
\label{lem:spherical-affine-model}
Let \(0\ne\mathcal C\subseteq W_K\) be a pointed polyhedral cone,
and let \(P,\ell,H_\ell,P_\ell\) be the associated objects defined
above.
Then \(P_\ell\) is a compact convex polytope of dimension
\(\dim\mathcal C-1\), and
\[
 \pi_\ell:P\longrightarrow P_\ell,
 \quad x\longmapsto\frac{x}{\ell(x)},
 \qquad
 \kappa_\ell:P_\ell\longrightarrow P,
 \quad y\longmapsto\frac{y}{\|y\|},
\]
are inverse homeomorphisms.  They induce a dimension-preserving
isomorphism of face posets, carry relative interiors to relative
interiors, and identify
face-to-face subdivisions of \(P\) with ordinary face-to-face
polytopal subdivisions of \(P_\ell\).
\end{lemma}

\begin{proof}
The displayed maps are well defined, continuous, and mutually inverse.
Since \(P\) is compact, \(P_\ell=\pi_\ell(P)\) is compact; being an
affine section of a polyhedral cone, it is therefore a compact convex
polytope.  Every nonzero \(x\in\mathcal C\) has the unique expression
\[
 x=\ell(x)y,
 \qquad
 y=\frac{x}{\ell(x)}\in P_\ell.
\]
This ray decomposition gives
\(\dim P_\ell=\dim\mathcal C-1\).

If \(\mathcal D\) is a nonzero face of \(\mathcal C\), then
\(\mathcal D\cap H_\ell\) is a face of \(P_\ell\).  Conversely, if
\(E\) is a face of \(P_\ell\), choose an affine functional on
\(H_\ell\) that is nonnegative on \(P_\ell\) and vanishes precisely
on \(E\).  Since \(\ell=1\) on \(H_\ell\), this affine functional is
the restriction of a linear functional \(\alpha\) on \(W_K\).
The ray decomposition shows that
\[
 \{ty:t\geq0,\ y\in E\}=\mathcal C\cap\ker\alpha,
\]
so this set is a face of \(\mathcal C\).  Thus
\(\mathcal D\mapsto\mathcal D\cap H_\ell\)
is a bijection from the nonzero faces of \(\mathcal C\) to the
nonempty faces of \(P_\ell\).  Under this bijection, the face of \(P\)
given by \(S(W_K)\cap\mathcal D\) corresponds to
\(\mathcal D\cap H_\ell\).  Since the relative interior of a face is
obtained by deleting all its proper faces, this correspondence also
preserves relative interiors.  This proves all the face-preservation
assertions.  Applying the same maps simultaneously to the cones of
the pieces proves the assertion about subdivisions.
\end{proof}

A finite collection \(\mathcal K\) of spherical polytopes is a
\emph{spherical polyhedral complex} if it contains every nonempty face
of each member and every nonempty intersection of two members is a
common face.  Distinct members of \(\mathcal K\) have disjoint relative
interiors.  It is a \emph{common refinement} of polytopes
\(P_1,\ldots,P_N\)
if, for every \(r\),
\[
 \{Q\in\mathcal K:Q\subseteq P_r,\ \dim Q=\dim P_r\}
\]
is a face-to-face subdivision of \(P_r\).

\subsection{Subdivision-invariant polyhedral chains}

Let
\[
 \widehat C_j^{\mathrm{poly}}(S(W_K);\F)
 =\F\langle\,\langle P\rangle:
       P\text{ is a spherical \(j\)-polytope}\,\rangle
\]
be the free vector space on symbols \(\langle P\rangle\), and impose
every
subdivision relation
\[
 \langle P\rangle=\sum_{\alpha=1}^s\langle P_\alpha\rangle.
\]
The quotient is denoted
\[
 C_j^{\mathrm{poly}}(S(W_K);\F).
\]
Since \(S(W_K)\) has dimension \(d\),
\[
 C_j^{\mathrm{poly}}(S(W_K);\F)=0\qquad(j>d).
\]
Let
\[
 \varpi_j:\widehat C_j^{\mathrm{poly}}(S(W_K);\F)
 \longrightarrow C_j^{\mathrm{poly}}(S(W_K);\F)
\]
be the quotient map.  For a polytope \(P\), write
\([P]=\varpi_j(\langle P\rangle)\) for its quotient class.

\begin{lemma}
\label{lem:spherical-common-refinement}
Every finite family of spherical polytopes has a finite common
refinement.  If the family is closed under \(P\mapsto-P\), the
refinement may be chosen invariant under the antipodal map.
\end{lemma}

\begin{proof}
The empty family requires no subdivision.  Otherwise, write
\(P_r=S(W_K)\cap\mathcal C_r\).  For each \(r\), choose a finite homogeneous
half-space representation by nonzero linear functionals,
\[
 \mathcal C_r=\{x\in W_K:\alpha_{r,s}(x)\geq0,\ 1\leq s\leq m_r\};
\]
an equality may be represented by the corresponding pair of opposite
inequalities.  Since \(\mathcal C_r\) is pointed,
\[
 \bigcap_s\ker\alpha_{r,s}
 =\mathcal C_r\cap(-\mathcal C_r)=\{0\}.
\]
In particular, the common intersection of all the hyperplanes
\(\ker\alpha_{r,s}\) is \(\{0\}\).  Their closed regions and all their
faces therefore form a finite polyhedral fan \(\Sigma\) of pointed
cones: if both \(x\) and \(-x\) belong to one region, then every
\(\alpha_{r,s}(x)\) vanishes, and hence \(x=0\).  On each cone of
\(\Sigma\), every
\(\alpha_{r,s}\) has a fixed weak sign.  Consequently, for every \(r\),
\[
 \mathcal C_r
 =\bigcup_{\substack{D\in\Sigma\\D\subseteq\mathcal C_r}}D,
\]
and the cones in this union meet face-to-face.  Thus they form a
subdivision of \(\mathcal C_r\).

Let
\[
 \mathcal K=
 \{S(W_K)\cap D:D\in\Sigma,\ D\ne\{0\},\
                    D\subseteq\mathcal C_r\text{ for some }r\}.
\]
If \(D\subseteq\mathcal C_r\), then
\[
 D\cap(-D)\subseteq
 \mathcal C_r\cap(-\mathcal C_r)=\{0\},
\]
so every member of \(\mathcal K\) is a spherical polytope.  Every
nonzero face of a retained cone, and every nonzero intersection of two
retained cones, is again retained.  Thus the fan property makes
\(\mathcal K\) a finite spherical polyhedral complex.  For each \(r\),
the spherical sections of the cones in \(\Sigma\) contained in
\(\mathcal C_r\)
form a face-to-face subdivision of \(P_r\).  Hence \(\mathcal K\) is a
common refinement of the given spherical polytopes.

The hyperplane arrangement, and hence \(\Sigma\), is invariant under
negation.  If the original family is closed under negation, then the
condition that a cone \(D\in\Sigma\) be contained in at least one
\(\mathcal C_r\) is also invariant under \(D\mapsto-D\).  The retained
spherical sections then form an antipodally invariant common
refinement.
\end{proof}

\begin{lemma}
\label{lem:refinement-membership}
Let \(P\) be a spherical \(j\)-polytope subdivided by a spherical
polyhedral complex \(\mathcal K\).

\begin{enumerate}[label=\textup{(\roman*)},leftmargin=*]
\item If \(Q\in\mathcal K\) has dimension \(j\) and
\(\relint Q\cap\relint P\ne\varnothing\), then \(Q\subseteq P\).

\item Suppose that a spherical polyhedral complex \(\mathcal L\)
refines \(\mathcal K\), meaning that it subdivides every member of
\(\mathcal K\).  If \(H\in\mathcal L\) and
\(Q\in\mathcal K\) have dimension \(j\), with
\[
 H\subseteq Q,
 \qquad
 \relint H\subseteq\relint Q,
\]
then
\[
 H\subseteq P\quad\Longleftrightarrow\quad Q\subseteq P.
\]
\end{enumerate}
\end{lemma}

\begin{proof}
For (i), choose
\(x\in\relint Q\cap\relint P\).  The subdivision of \(P\) contains a
member \(Q'\in\mathcal K\), with \(Q'\subseteq P\), whose relative
interior contains \(x\).  Relative interiors of members of
\(\mathcal K\) are disjoint, so \(Q'=Q\).

For (ii), only the forward implication needs proof.  If
\(H\subseteq P\), then the common dimension implies that
\(\relint H\) meets \(\relint P\).  Any point of this intersection
lies in \(\relint Q\), so (i) gives \(Q\subseteq P\).
\end{proof}

Let
\[
 z=\sum_{r=1}^Na_r\langle P_r\rangle
   \in\widehat C_j^{\mathrm{poly}}(S(W_K);\F),
 \qquad a_r\in\F,
\]
and let \(\mathcal K\) be a common refinement of the polytopes in
\(z\).  For each \(j\)-dimensional member \(Q\in\mathcal K\), define
\[
 \mu_z(Q)=\sum_{\{r:Q\subseteq P_r\}}a_r\in\F.
\]

\begin{lemma}\label{lem:coefficient-criterion}
Two free \(j\)-chains \(z,z'\) have the same quotient image precisely
when, on some (equivalently, every) common refinement \(\mathcal K\),
\[
 \mu_z(Q)=\mu_{z'}(Q)\qquad(Q\in\mathcal K,\ \dim Q=j).
\]
\end{lemma}

\begin{proof}
Indeed, if the displayed equality holds, subdivision gives
\[
 \varpi_j(z)
     =\sum_{\substack{Q\in\mathcal K\\\dim Q=j}}\mu_z(Q)[Q]
     =\sum_{\substack{Q\in\mathcal K\\\dim Q=j}}\mu_{z'}(Q)[Q]
     =\varpi_j(z').
\]
Conversely, fix an arbitrary common refinement \(\mathcal K\) and
suppose \(\varpi_j(z)=\varpi_j(z')\).  Express \(z+z'\) as a finite
linear combination
of subdivision relations, and apply
Lemma~\ref{lem:spherical-common-refinement} to the members of
\(\mathcal K\) and to every polytope occurring in those relations.
Let \(\mathcal L\) be the resulting common refinement.  In a relation
\[
 \langle P\rangle+\sum_\alpha\langle P_\alpha\rangle,
\]
a \(j\)-dimensional member \(H\in\mathcal L\) contributes nothing if
\(H\not\subseteq P\).  If \(H\subseteq P\), choose a point of
\(\relint H\) outside the union of the boundaries of the
top-dimensional pieces \(P_\alpha\).  It lies in the relative interior
of a unique \(P_\alpha\), and
Lemma~\ref{lem:refinement-membership}(i), applied to the
\(\mathcal L\)-subdivision of \(P_\alpha\), gives
\(H\subseteq P_\alpha\).  No second piece can contain \(H\), since
distinct top-dimensional subdivision pieces meet in a proper face.
Thus the relation has coefficient \(1+1=0\) on every such \(H\), and
therefore
\[
 \mu_z(H)=\mu_{z'}(H)
 \qquad(H\in\mathcal L,\ \dim H=j).
\]

 Given a \(j\)-dimensional member \(Q\in\mathcal K\), choose a
 \(j\)-dimensional member
 \(H\in\mathcal L\) such that
 \[
  H\subseteq Q,\qquad \dim H=\dim Q=j,\qquad
  \relint H\subseteq\relint Q.
 \]
 For every original \(j\)-polytope \(P\) occurring in \(z\) or
 \(z'\), Lemma~\ref{lem:refinement-membership}(ii) gives
 \(H\subseteq P\) if and only if \(Q\subseteq P\).  Therefore
\[
 \mu_z(Q)=\mu_z(H)=\mu_{z'}(H)=\mu_{z'}(Q).
\]
\end{proof}

In particular, every nonempty spherical polytope satisfies
\[
 [P]\ne0\quad\text{in }C_{\dim P}^{\mathrm{poly}}(S(W_K);\F).
\]
Indeed, take a common refinement of the free chains
\(\langle P\rangle\) and \(0\),
and choose a top-dimensional refinement member \(Q\subseteq P\).
Then
\(\mu_{\langle P\rangle}(Q)=1\), whereas \(\mu_0(Q)=0\).  The
coefficient criterion therefore gives \([P]\ne0\).

The zero cone corresponds to the empty spherical face and is not a
chain generator; in particular, a spherical \(0\)-polytope has no
nonempty facets.  Set \(C_{-1}^{\mathrm{poly}}=0\) and
\(\partial_0^{\mathrm{poly}}=0\).  For \(j\geq1\), first define on the
free space
\[
 \widehat\partial_j^{\mathrm{poly}}\langle P\rangle
 =\sum_{F\lessdot P}[F]
 \in C_{j-1}^{\mathrm{poly}}(S(W_K);\F).
\]

\begin{proposition}
\label{prop:polyhedral-boundary}
The map \(\widehat\partial_j^{\mathrm{poly}}\) respects every
subdivision relation and hence induces
\[
 \partial_j^{\mathrm{poly}}:
 C_j^{\mathrm{poly}}(S(W_K);\F)
 \longrightarrow C_{j-1}^{\mathrm{poly}}(S(W_K);\F).
\]
For every \(j\geq1\),
\[
 \partial_{j-1}^{\mathrm{poly}}\partial_j^{\mathrm{poly}}=0.
\]
\end{proposition}

\begin{proof}
Let \(P=S(W_K)\cap\mathcal C\) and let
\(P=\bigcup_{\alpha=1}^sP_\alpha\) be a face-to-face subdivision.
Let \(\ell,P_\ell\) be defined in
Lemma~\ref{lem:spherical-affine-model}, then \(\pi_\ell\) identifies
the given subdivision, together with all its face incidences, with an
ordinary face-to-face polytopal subdivision of \(P_\ell\).

Returning to the original spherical subdivision, write
\(\mathcal B^\circ\) for the set of its \((j-1)\)-faces whose relative
interiors lie in \(\relint P\), and, for each facet \(F\lessdot P\),
write \(\mathcal B_F\) for the \((j-1)\)-faces that are contained in
\(F\).
In an ordinary polytopal subdivision, every internal facet belongs to
exactly two top-dimensional pieces, every boundary facet belongs to
exactly one, and the subdivision restricts to a face-to-face
subdivision of each facet.  Pulling these facts back through
\(\kappa_\ell\) gives
\[
 \begin{aligned}
 \sum_{\alpha=1}^s
   \widehat\partial_j^{\mathrm{poly}}\langle P_\alpha\rangle
 &=2\sum_{G\in\mathcal B^\circ}[G]
   +\sum_{F\lessdot P}\ \sum_{G\in\mathcal B_F}[G]\\
 &=\sum_{F\lessdot P}[F]
  =\widehat\partial_j^{\mathrm{poly}}\langle P\rangle.
 \end{aligned}
\]
Thus the boundary descends to the quotient.  Suppose \(j\geq2\), and
let \(G\) be a \((j-2)\)-face of \(P\).  The face correspondence in
Lemma~\ref{lem:spherical-affine-model} identifies the interval
\([G,P]\) with the corresponding interval in the face lattice of
\(P_\ell\).  By the diamond property for polytope face lattices
\cite[Theorem~2.7(iii)]{Ziegler},
\[
 \#\{F:G\lessdot F\lessdot P\}=2
\]
and hence
\[
 \partial_{j-1}^{\mathrm{poly}}
 \partial_j^{\mathrm{poly}}[P]
 =\sum_{\substack{G\text{ a }(j-2)\text{-face}\\\text{of }P}}2[G]=0.
\]
For \(j=1\), the identity is \(\partial_0^{\mathrm{poly}}=0\).
\end{proof}

The antipodal map induces
\[
 (A_{W_K})_{\#,j}:
 C_j^{\mathrm{poly}}(S(W_K);\F)
 \longrightarrow C_j^{\mathrm{poly}}(S(W_K);\F),
 \qquad
 (A_{W_K})_{\#,j}[P]=[-P].
\]
This map commutes with the polyhedral boundary.

\begin{proposition}\label{prop:polyhedral-augmentation}
The formula
\[
 \epsilon_{\mathrm{poly}}\!
 \left(\sum_p a_p[p]\right)=\sum_p a_p
\]
defines an augmentation on
\(C_*^{\mathrm{poly}}(S(W_K);\F)\).  It is preserved by the antipodal
involution.
\end{proposition}

\begin{proof}
A spherical \(0\)-polytope is a point and has no nontrivial
subdivision, so the displayed formula is well defined.  By
Lemma~\ref{lem:spherical-affine-model}, every spherical
\(1\)-polytope is face-preservingly homeomorphic to a compact
one-dimensional convex polytope, hence to a closed line segment.  It
therefore has exactly two vertices, and
\(\epsilon_{\mathrm{poly}}\partial_1^{\mathrm{poly}}=0\) over \(\F\).
The antipodal map merely permutes the points, and hence preserves the
augmentation.
\end{proof}

\begin{lemma}\label{lem:polyhedral-antipodal-exactness}
For \(0\leq j\leq d\), put
\[
 \nu_j=\id+(A_{W_K})_{\#,j}.
\]
We call \(\nu_j\) the antipodal norm operator.  Then
\[
 \ker\nu_j=\operatorname{im}\nu_j.
\]
\end{lemma}

\begin{proof}
The inclusion \(\operatorname{im}\nu_j\subseteq\ker\nu_j\) follows
from \(\nu_j^2=0\).  Conversely, let \(Z\in\ker\nu_j\) and choose a
free representative
\[
 z=\sum_{r=1}^N a_r\langle P_r\rangle
\]
of \(Z\).  Apply Lemma~\ref{lem:spherical-common-refinement} to the
family consisting of the \(P_r\) and the \(-P_r\), and choose the
resulting common refinement \(\mathcal K\) to be antipodally
invariant.  Subdivision gives
\[
 Z=\sum_{\substack{Q\in\mathcal K\\\dim Q=j}}\mu_z(Q)[Q].
\]
Let
\[
 z^-=\sum_{r=1}^N a_r\langle-P_r\rangle.
\]
Then
\(\varpi_j(z^-)=(A_{W_K})_{\#,j}Z\), and \(\mathcal K\) is a common
refinement of the polytopes occurring in \(z\) and \(z^-\).  Moreover,
\[
 \mu_{z^-}(Q)=\mu_z(-Q)
 \qquad(Q\in\mathcal K,\ \dim Q=j).
\]
Since \(Z\in\ker\nu_j\), one has
\((A_{W_K})_{\#,j}Z=Z\), and hence
\(\varpi_j(z^-)=\varpi_j(z)\).
Lemma~\ref{lem:coefficient-criterion} therefore yields
\[
 \mu_z(Q)=\mu_z(-Q)
 \qquad(Q\in\mathcal K,\ \dim Q=j).
\]
No nonempty member \(Q\) of \(\mathcal K\) equals \(-Q\).  Indeed, if
\(x\in Q=-Q\), then both \(x\) and \(-x\) belong to the defining cone
of \(Q\), contrary to pointedness.  Let \(\mathcal R_j\) contain
exactly one member of each antipodal orbit on the \(j\)-dimensional
members of \(\mathcal K\), and put
\[
 W=\sum_{Q\in\mathcal R_j}\mu_z(Q)[Q].
\]
Then
\[
 \begin{aligned}
 \nu_jW
 &=\sum_{Q\in\mathcal R_j}
   \mu_z(Q)\bigl([Q]+[-Q]\bigr)\\
 &=\sum_{\substack{Q\in\mathcal K\\\dim Q=j}}
   \mu_z(Q)[Q]
 =Z.
 \end{aligned}
\]
This shows that $Z\in\operatorname{im}\nu_j$.
\end{proof}

\section{From a rook labeling to a polyhedral chain map}

\subsection{Rook galleries and rank}

For \((a,b)\in\Omega_K\), define
\begin{equation*}
 \operatorname{root}(a,b)=e_a-e_b\in W_K.
\end{equation*}
Here \(e_1,\ldots,e_K\) are the standard basis vectors of
\(\mathbb R^K\).
All ranks and linear-independence statements concerning roots are over
\(\mathbb R\).
A sequence \((a_i,b_i)_{i=1}^m\) with \(m\geq1\) is a
\emph{rook gallery} if every two
consecutive terms have the same first coordinate or the same second
coordinate.

\begin{lemma}
\label{lem:gallery-rank}
Let \((a_i,b_i)_{i=1}^m\) be a rook gallery. If
\[
 0\in\conv\{e_{a_i}-e_{b_i}:1\leq i\leq m\},
\]
then
\[
 \rank\{e_{a_i}-e_{b_i}:1\leq i\leq m\}\leq m-2.
\]
\end{lemma}

\begin{proof}
Let \(G\) be the directed multigraph with edges \(a_i\to b_i\), and
put \(v=|V(G)|\). Split each used coordinate symbol \(s\) into a tail
vertex \(s^+\) and a head vertex \(s^-\), retaining only the roles in
which \(s\) occurs, and replace \(a_i\to b_i\) by the bipartite edge
\(a_i^+b_i^-\). Consecutive gallery edges meet at the copy of their
common coordinate, so this split graph is connected. In particular,
the underlying graph of \(G\) is connected.

Let \(t\) be the number of coordinate symbols that occur both as tails
and as heads. The split graph has \(v+t\) vertices and \(m\) edges.
Its connectedness gives
\begin{equation}\label{eq:split-gallery-count}
 m\geq v+t-1.
\end{equation}

There are coefficients satisfying
\[
 \lambda_i\geq0,\qquad
 \sum_{i=1}^m\lambda_i=1,
 \qquad
 \sum_{i=1}^m\lambda_i(e_{a_i}-e_{b_i})=0.
\]
The last equality is flow conservation at every vertex. Start with
an edge of positive weight. At its head, conservation supplies an
outgoing edge of positive weight; iterating in the finite graph
eventually repeats a vertex and hence produces a directed cycle.
Because \(a_i\ne b_i\), this cycle is not a loop, so at least two
coordinate symbols on it occur in both roles. Hence \(t\geq2\).
Incidence vectors of a connected graph on \(v\) vertices have rank
\(v-1\). By \eqref{eq:split-gallery-count},
\[
 \rank\{e_{a_i}-e_{b_i}:1\leq i\leq m\}
 =v-1\leq m-t\leq m-2.
\]
\end{proof}

For the rest of the construction, suppose that
\begin{equation*}
 q:V(Q_K)\longrightarrow\Omega_K
\end{equation*}
is a rook labeling.

\subsection{Freudenthal subdivision}

The cubical cellular complex is retained as the source because the
rook condition controls labels only along cube edges.  Although every
maximal Freudenthal simplex in a cubical face follows a monotone cube
path, one of its lower-dimensional faces may skip intermediate
vertices and hence contain an edge joining nonadjacent cube vertices.
The rook condition gives no direct relation between the labels at the
ends of such an edge.  Passing from a cubical face to the sum of all
its maximal Freudenthal simplices makes these uncontrolled internal
faces cancel over \(\F\).

Give \(I^K=[0,1]^K\) its standard cubical structure and write
\[
 \partial_{\mathrm{top}}I^K
 =\{x\in I^K:x_i\in\{0,1\}\text{ for at least one }i\}.
\]
This is the topological boundary with its induced cell structure.  The
vertex antipode extends to the cellular involution
\[
 A_K(x_1,\ldots,x_K)=(1-x_1,\ldots,1-x_K).
\]

Let \({\mathcal T}_K\) be the standard Freudenthal
triangulation of \(I^K\). If a cubical \(j\)-face \(F\) has lower vertex
\(0_F\), upper vertex \(1_F\), and free coordinates
\(i_1<\cdots<i_j\), then its maximal simplices are
\[
 \sigma_\pi=[x_0,\ldots,x_j],
 \qquad \pi\in\mathfrak S_j,
\]
where \(x_r\) is obtained from \(0_F\) by changing
\(i_{\pi(1)},\ldots,i_{\pi(r)}\) from \(0\) to \(1\).

We use its
mod--\(2\) simplicial chains, with
\[
 \partial_j^{\mathrm{simp}}[x_0,\ldots,x_j]
 =\sum_{i=0}^j[x_0,\ldots,\widehat{x_i},\ldots,x_j].
\]
Here \(C_{-1}^{\mathrm{simp}}=0\), \(\partial_0^{\mathrm{simp}}=0\),
and
\((A_K)_{\#,j}[x_0,\ldots,x_j]=[A_Kx_0,\ldots,A_Kx_j]\).

Define an \(\F\)-linear map
\[
 \mathfrak F_j:
 C_j^{\cell}(\partial_{\mathrm{top}}I^K;\F)
 \longrightarrow C_j^{\mathrm{simp}}(\mathcal T_K;\F)
\]
on a cubical \(j\)-face by
\begin{equation}\label{eq:freudenthal-chain}
 \mathfrak F_j(F)=\sum_{\pi\in\mathfrak S_j}\sigma_\pi.
\end{equation}
Deleting an internal \(x_r\) pairs the resulting facet with the order
obtained by swapping \(\pi(r)\) and \(\pi(r+1)\), so such facets cancel.
Deleting \(x_0\) or \(x_j\) gives, respectively, a maximal simplex in
the corresponding upper or lower cubical facet, and every boundary
simplex occurs once.  Hence
\begin{equation}\label{eq:freudenthal-boundary}
 \partial_j^{\mathrm{simp}}\mathfrak F_j
 =\mathfrak F_{j-1}\partial_j^{\cell}\qquad(j\geq1).
\end{equation}
The antipode reverses the vertex order of each monotone gallery, hence
permutes the summands in \eqref{eq:freudenthal-chain} and gives
\begin{equation}\label{eq:freudenthal-equivariance}
 \mathfrak F_j(A_KF)=(A_K)_\#\mathfrak F_j(F).
\end{equation}

\subsection{Radial chains}

Let \(j\geq0\), and let
\(v_0,\ldots,v_j\in W_K\setminus\{0\}\) satisfy
\[
 0\notin\conv\{v_0,\ldots,v_j\}.
\]
Write
\[
\mathcal C=\pos\{v_0,\ldots,v_j\}
 =\left\{\sum_{i=0}^jt_iv_i:t_i\geq0\right\}
\]
for their positive cone.  This cone is pointed: a nonzero vector in its
intersection with its negative would give a nontrivial nonnegative
relation among the \(v_i\), contrary to the convex-hull hypothesis.
When the \(v_i\) are linearly independent, the cone has dimension
\(j+1\), so its spherical section is a spherical \(j\)-polytope.
Define the \emph{radial chain}
\(\mathcal R[v_0,\ldots,v_j]\in
C_j^{\mathrm{poly}}(S(W_K);\F)\) by
\[
 \mathcal R[v_0,\ldots,v_j]=
 \begin{cases}
 [S(W_K)\cap\pos\{v_0,\ldots,v_j\}],
     &v_0,\ldots,v_j\text{ are linearly independent},\\
 0,&v_0,\ldots,v_j\text{ are linearly dependent}.
 \end{cases}
\]
By the nonvanishing observation above,
\[
 \mathcal R[v_0,\ldots,v_j]\ne0
 \quad\Longleftrightarrow\quad
 v_0,\ldots,v_j\text{ are linearly independent}.
\]
Both linear dependence and the positive cone are unchanged by
permuting \(v_0,\ldots,v_j\).  Hence
\(\mathcal R[v_0,\ldots,v_j]\) is invariant under every permutation of
its entries.

Let \(\ell,H_\ell,\kappa_\ell\) be defined as in Lemma~\ref{lem:spherical-affine-model},
and put
\[
 w_i=\frac{v_i}{\ell(v_i)}\qquad(0\leq i\leq j).
\]

\begin{lemma}
\label{lem:positive-cone-section}
For every nonempty \(I\subseteq\{0,\ldots,j\}\),
\[
 H_\ell\cap\pos\{v_i:i\in I\}=\conv\{w_i:i\in I\},
 \qquad
 \kappa_\ell\bigl(\conv\{w_i:i\in I\}\bigr)
 =S(W_K)\cap\pos\{v_i:i\in I\}.
\]
Moreover, \(\{v_i:i\in I\}\) is linearly independent if and only if
\(\{w_i:i\in I\}\) is affinely independent.
\end{lemma}

\begin{proof}
If \(x=\sum_{i\in I}t_iv_i\in H_\ell\), then
\(s_i=t_i\ell(v_i)\) are nonnegative, satisfy
\(\sum_{i\in I}s_i=1\), and give \(x=\sum_{i\in I}s_iw_i\); the
converse follows by taking \(t_i=s_i/\ell(v_i)\).  This proves the
first identity, and radial normalization gives the second.  Since
\(\ell(w_i)=1\), the rescaling \(d_i=c_i\ell(v_i)\) bijects linear
relations among the \(v_i\) with affine relations among the \(w_i\).
\end{proof}

For independent generators, the radial boundary is the usual
spherical-simplex boundary.  The difficulty is the dependent case:
the radial \(j\)-chain vanishes, but its codimension-one radial faces
need not vanish separately.

Fix \(j\geq1\) and points \(w_0,\ldots,w_j\) in an affine space such
that
\[
 \dim\conv\{w_0,\ldots,w_j\}=j-1.
\]
Let
\[
 \Delta^j
 =
 \left\{t=(t_0,\ldots,t_j)\in\mathbb R^{j+1}:
 t_i\geq0,\ \sum_{i=0}^j t_i=1\right\},
\]
and let
\[
 f:\Delta^j\longrightarrow
 \operatorname{aff}\{w_0,\ldots,w_j\},
 \qquad
 f(t)=\sum_{i=0}^j t_iw_i.
\]
For \(0\leq i\leq j\), write
\[
 \Delta_i=\{t\in\Delta^j:t_i=0\},
 \qquad
 F_i=f(\Delta_i)
     =\conv\{w_0,\ldots,\widehat{w_i},\ldots,w_j\}.
\]
Finally, put
\[
 E=f\bigl((\Delta^j)^{(j-2)}\bigr),
\]
where \((\Delta^j)^{(j-2)}\) is the union of the faces spanned by at
most \(j-1\) vertices; for \(j=1\), this set is empty.

Carath\'eodory's theorem gives
\begin{equation}\label{eq:low-dimensional-facet-image}
 \dim F_i\leq j-2
 \quad\Longrightarrow\quad
 F_i\subseteq E.
\end{equation}
Indeed, every point of such an \(F_i\) is a convex combination of at
most \(j-1\) of its vertices.  Thus, for \(y\notin E\), every \(F_i\)
containing \(y\) is automatically \((j-1)\)-dimensional.

\begin{lemma}
\label{lem:generic-degenerate-fiber}
For every
\[
 y\in\operatorname{aff}\{w_0,\ldots,w_j\}\setminus E,
\]
exactly two of the sets \(F_i\) contain \(y\) when
\(y\in f(\Delta^j)\), and none contains \(y\) otherwise.  Consequently,
\[
 \#\{i:y\in F_i,\ \dim F_i=j-1\}=0
 \qquad\text{in }\F.
\]
\end{lemma}

\begin{proof}
Let
\[
 A^j=\left\{t\in\mathbb R^{j+1}:\sum_{i=0}^j t_i=1\right\},
\]
and let
\[
 \widetilde f:A^j\longrightarrow
 \operatorname{aff}\{w_0,\ldots,w_j\}
\]
be the affine extension of \(f\).  Since
\(\dim A^j=j\) and the image has dimension \(j-1\), every fiber of
\(\widetilde f\) is an affine line.

If \(y\notin f(\Delta^j)\), then no \(F_i\) contains \(y\).  Suppose
that \(y\in f(\Delta^j)\), and set
\[
 \Lambda_y=\widetilde f^{-1}(y),
 \qquad
 L_y=\Lambda_y\cap\Delta^j.
\]
Then \(L_y\) is a nonempty compact interval and
\[
 L_y\cap(\Delta^j)^{(j-2)}=\varnothing.
\]

The interval \(L_y\) is nondegenerate.  Indeed, suppose that
\(L_y=\{t\}\), and write
\(\Lambda_y=t+\mathbb R\xi\) with \(\xi\ne0\).
Since \(t\) does not lie in the \((j-2)\)-skeleton, at most one
coordinate of \(t\) vanishes.  If all coordinates are positive, then
\(t+\varepsilon\xi\in\Delta^j\) for all sufficiently small
\(\varepsilon\).  If \(t_r=0\) is the unique zero coordinate, choose
the sign of a sufficiently small \(\varepsilon\ne0\) so that
\(\varepsilon\xi_r\geq0\); again
\(t+\varepsilon\xi\in\Delta^j\).  Both cases contradict
\(L_y=\{t\}\).

Write
\[
 L_y=[p,q],
 \qquad p\ne q.
\]
No facet \(\Delta_i\) contains \(L_y\).  Otherwise the restriction of
\(\widetilde f\) to \(\operatorname{aff}\Delta_i\) would be
noninjective, and hence
\[
 \dim F_i\leq j-2.
\]
Since \(y\in F_i\), this would contradict
\eqref{eq:low-dimensional-facet-image} and \(y\notin E\).

Because \(L_y\) avoids the \((j-2)\)-skeleton, each endpoint lies in
the relative interior of a unique facet.  These facets are distinct:
otherwise their convexity would imply that they contain all of
\(L_y\), contrary to the preceding paragraph.

Moreover,
\[
 \operatorname{relint}L_y\subseteq\operatorname{relint}\Delta^j.
\]
Indeed, if an interior point of \(L_y\) belonged to \(\Delta_i\), then
the nonnegative affine function \(t\mapsto t_i\) on \(L_y\) would
vanish at an interior point and hence vanish identically.  This would
force \(L_y\subseteq\Delta_i\), which was just ruled out.

Finally,
\[
 y\in F_i
 \quad\Longleftrightarrow\quad
 L_y\cap\Delta_i\ne\varnothing.
\]
Since \(L_y\) meets the boundary of \(\Delta^j\) only at \(p\) and
\(q\), exactly two of the sets \(F_i\) contain \(y\).  By
\eqref{eq:low-dimensional-facet-image}, both have dimension \(j-1\).
Their total multiplicity is therefore \(2=0\) in \(\F\).
\end{proof}

\begin{lemma}
\label{lem:degenerate-radial-cancellation}
Let \(j\geq1\).  Suppose that
\(v_0,\ldots,v_j\in W_K\setminus\{0\}\) are linearly dependent and
\[
 0\notin\conv\{v_0,\ldots,v_j\}.
\]
Then
\[
 \sum_{i=0}^j
 \mathcal R[v_0,\ldots,\widehat{v_i},\ldots,v_j]=0.
\]
\end{lemma}

\begin{proof}
The cone
\[
 \mathcal C=\pos\{v_0,\ldots,v_j\}
\]
is pointed by the convex-hull hypothesis.  Let
\(\ell,\kappa_\ell\) be defined as in
Lemma~\ref{lem:spherical-affine-model}.  Put
\[
 w_i=\frac{v_i}{\ell(v_i)}.
\]
Lemma~\ref{lem:positive-cone-section} shows that
\(\{w_0,\ldots,w_j\}\) is affinely dependent.  If
\[
 \dim\conv\{w_0,\ldots,w_j\}\leq j-2,
\]
then every \(j\)-element subfamily of the \(w_i\) is affinely
dependent.  The same lemma shows that every radial chain in the
asserted sum is zero.

It remains to consider
\(\dim\conv\{w_0,\ldots,w_j\}=j-1\).  Use the notation
\(f,\Delta^j,F_i,E\) from
Lemma~\ref{lem:generic-degenerate-fiber}, and define free chains
\[
 \widehat P_i=
 \begin{cases}
  \langle\kappa_\ell(F_i)\rangle,&\dim F_i=j-1,\\
  0,&\dim F_i\leq j-2,
 \end{cases}
 \qquad
 P_i=\varpi_{j-1}(\widehat P_i).
\]
By Lemma~\ref{lem:positive-cone-section},
\[
 P_i=\mathcal R[v_0,\ldots,\widehat{v_i},\ldots,v_j].
\]

Choose a common refinement \(\mathcal K\) of the spherical polytopes
occurring in the nonzero \(\widehat P_i\).  Let
\(Q\in\mathcal K\) have dimension \(j-1\).  If \(Q\) is contained in
none of these polytopes, then
\(\mu_{\sum_i\widehat P_i}(Q)=0\).  We may therefore assume that
\(Q\) is contained in at least one of them.  Then
\(\kappa_\ell^{-1}(\relint Q)\) lies in
\(\operatorname{aff}\{w_0,\ldots,w_j\}\).  This relative interior has
dimension \(j-1\), whereas \(E\) is a finite union of polytopes of
dimension at most \(j-2\).  Choose
\[
 y\in\kappa_\ell^{-1}(\relint Q)\setminus E,
 \qquad
 z=\kappa_\ell(y)\in\relint Q.
\]
Suppose \(y\in F_i\).  By
\eqref{eq:low-dimensional-facet-image}, \(y\notin E\) forces
\(\dim F_i=j-1\).  Thus \(F_i\) is a \((j-1)\)-simplex and
\[
 \partial F_i\subseteq f((\Delta^j)^{(j-2)})=E.
\]
Hence \(y\in\relint F_i\).
Lemma~\ref{lem:positive-cone-section}, applied with
\(I=\{0,\ldots,j\}\setminus\{i\}\), gives
\(z=\kappa_\ell(y)\in\relint\kappa_\ell(F_i)\).
Since \(\mathcal K\) subdivides
\(\kappa_\ell(F_i)\), Lemma~\ref{lem:refinement-membership}(i) gives
\[
 Q\subseteq\kappa_\ell(F_i)
 \quad\Longleftrightarrow\quad
 y\in F_i.
\]
The reverse implication in this display is the lemma; the forward
implication is immediate from \(y=\kappa_\ell^{-1}(z)\).

It follows from Lemma~\ref{lem:generic-degenerate-fiber} that
\[
 \mu_{\sum_i\widehat P_i}(Q)
 =
 \#\{i:y\in F_i,\ \dim F_i=j-1\}=0
 \quad\text{in }\F.
\]
The coefficient criterion gives
\[
 \sum_iP_i=\varpi_{j-1}\!\left(\sum_i\widehat P_i\right)=0,
\]
which is the desired cancellation.
\end{proof}

\begin{corollary}
\label{coro:radial-simplex-identities}
Let \(j\geq0\), and let
\(v_0,\ldots,v_j\in W_K\setminus\{0\}\) satisfy
\[
 0\notin\conv\{v_0,\ldots,v_j\}.
\]
If \(j\geq1\), then
\begin{equation}\label{eq:radial-simplex-boundary}
 \partial_j^{\mathrm{poly}}\mathcal R[v_0,\ldots,v_j]
 =\sum_{i=0}^j
   \mathcal R[v_0,\ldots,\widehat{v_i},\ldots,v_j].
\end{equation}
In every degree \(j\geq0\),
\begin{equation}\label{eq:radial-antipode}
 (A_{W_K})_{\#,j}\mathcal R[v_0,\ldots,v_j]
 =\mathcal R[-v_0,\ldots,-v_j].
\end{equation}
\end{corollary}

\begin{proof}
Equation~\eqref{eq:radial-antipode} follows directly from
\(-\pos\{v_i\}=\pos\{-v_i\}\).  For the boundary identity, suppose
first that the \(v_i\) are linearly independent.  Put
\(\mathcal C=\pos\{v_0,\ldots,v_j\}\) and set
\(w_i=v_i/\ell(v_i)\).  By
Lemmas~\ref{lem:spherical-affine-model}
and~\ref{lem:positive-cone-section}, the spherical section is
face-preservingly homeomorphic to the simplex
\(\conv\{w_0,\ldots,w_j\}\).  Its facets therefore correspond to the
subfamilies obtained by omitting one generator.
Proposition~\ref{prop:polyhedral-boundary} gives
\eqref{eq:radial-simplex-boundary}.  If the \(v_i\) are
dependent, the left-hand side is zero by definition and the right-hand
side is zero by
Lemma~\ref{lem:degenerate-radial-cancellation}.
\end{proof}

\subsection{The rook polyhedral chain}

For \(0\leq j\leq K-1\), define an \(\F\)-linear map
\[
 \rho_j:
 C_j^{\mathrm{simp}}(\mathcal T_K;\F)
 \longrightarrow C_j^{\mathrm{poly}}(S(W_K);\F)
\]
as follows.  For
\(\sigma=[x_0,\ldots,x_j]\in\mathcal T_K\), put
\[
 v_i=\operatorname{root}(q(x_i))\qquad(0\leq i\leq j)
\]
and set
\[
 \rho_j(\sigma)=
 \begin{cases}
  \mathcal R[v_0,\ldots,v_j],&
      0\notin\conv\{v_0,\ldots,v_j\},\\
 0,&0\in\conv\{v_0,\ldots,v_j\}.
 \end{cases}
\]
Both cases are invariant under permutations of the vertices of
\(\sigma\), so \(\rho_j\) is well defined on mod--\(2\) simplicial
chains.  We do not claim that the maps \(\rho_j\) form a chain map on
all of \(C_*^{\mathrm{simp}}(\mathcal T_K;\F)\); the boundary identity
will be proved only on the maximal gallery simplices used by
\(\mathfrak F_*\).

Define the composite \(\Theta_j\) and record the source and target of
each step as follows:
\begin{equation}\label{eq:rook-polyhedral-composition}
 C_j^{\cell}(\partial_{\mathrm{top}}I^K;\F)
 \xrightarrow{\ \mathfrak F_j\ }
 C_j^{\mathrm{simp}}(\mathcal T_K;\F)
 \xrightarrow{\ \rho_j\ }
 C_j^{\mathrm{poly}}(S(W_K);\F),
 \qquad
 \Theta_j=\rho_j\circ\mathfrak F_j.
\end{equation}
Thus, for a cubical \(j\)-face \(F\),
\[
 \Theta_j(F)=\rho_j(\mathfrak F_j(F)).
\]

\begin{lemma}
\label{lem:rook-polyhedral-boundary}
Let \(1\leq j\leq K-1\), let \(F\) be a cubical \(j\)-face, and let
\(\sigma\) be a maximal Freudenthal \(j\)-simplex of \(F\).  Then
\begin{equation}\label{eq:R-chain-map}
 \partial_j^{\mathrm{poly}}\rho_j(\sigma)
 =\rho_{j-1}(\partial_j^{\mathrm{simp}}\sigma).
\end{equation}
Consequently, every cubical \(j\)-face \(F\) satisfies
\begin{equation}\label{eq:theta-boundary}
 \partial_j^{\mathrm{poly}}\Theta_j(F)
 =\Theta_{j-1}(\partial_j^{\cell}F).
\end{equation}
\end{lemma}

\begin{proof}
The vertices of \(\sigma\) form a monotone cube path, so their
labels form a rook gallery.
If its root convex hull avoids the origin,
\eqref{eq:R-chain-map} is exactly
\eqref{eq:radial-simplex-boundary}.  Suppose the convex hull contains the
origin.  Then \(\rho_j(\sigma)=0\), so the left-hand side is zero.
With \(m=j+1\), Corollary~\ref{lem:gallery-rank} gives root rank at most
\(j-1\).  When \(j=1\), this is impossible because every root is
nonzero.  Assume therefore that \(j\geq2\).  A facet has \(j\) roots.
If its convex hull contains the origin, its \(\rho_{j-1}\)-chain is
zero by definition; otherwise those \(j\) roots are linearly dependent,
since their rank is at most \(j-1\), so their radial chain is again zero
by definition.  Thus both sides of
\eqref{eq:R-chain-map} vanish.

Summing \eqref{eq:R-chain-map} over the maximal simplices in
\eqref{eq:freudenthal-chain} and using
\eqref{eq:freudenthal-boundary} gives the consequence.
\end{proof}

\begin{proposition}\label{prop:polyhedral-rook-chain}
The maps \(\Theta_j\) in
\eqref{eq:rook-polyhedral-composition}, extended by zero outside the
dimensions of the source, form an equivariant,
augmentation-preserving chain map
\[
 \Theta_*:
 C_*^{\cell}(\partial_{\mathrm{top}}I^K;\F)
 \longrightarrow
 C_*^{\mathrm{poly}}(S(W_K);\F).
\]
\end{proposition}

\begin{proof}
Equation~\eqref{eq:theta-boundary} proves the chain-map identity in
every positive degree, including the top source degree
\(d+1=K-1\).  In that degree,
\[
 C_{d+1}^{\mathrm{poly}}(S(W_K);\F)=0,
 \qquad
 \Theta_{d+1}=0,
\]
so \eqref{eq:theta-boundary} reads
\[
 \Theta_d\partial_{d+1}^{\cell}=0.
\]

If \(x\) is a source vertex and \(q(x)=(a,b)\), then
\[
 \Theta_0(x)
 =\left[\left\{\frac{e_a-e_b}{\|e_a-e_b\|}\right\}\right].
\]
Thus \(\epsilon_{\mathrm{poly}}\Theta_0(x)=1\), proving preservation
of the augmentation.

Finally,
\[
 \operatorname{root}(q(A_Kx))
 =\operatorname{root}(\tau q(x))
 =-\operatorname{root}(q(x)).
\]
Negation preserves the condition that a root convex hull contain the
origin.  On the complementary case, the radial antipode identity
\eqref{eq:radial-antipode}, together with invariance of radial chains
under permutations of their entries, shows that \(\rho_*\) intertwines
the two antipodal involutions.  Combining this with
\eqref{eq:freudenthal-equivariance} proves
\[
 \Theta_j(A_KF)=(A_{W_K})_{\#,j}\Theta_j(F)
\]
for every cubical face \(F\).
\end{proof}

\begin{proof}[Proof of Theorem~\ref{thm:diagonal-rook}]
Suppose that a rook labeling \(q:V(Q_K)\to\Omega_K\) exists.  Put
\(d=K-2\).  Proposition~\ref{prop:polyhedral-rook-chain} gives an
equivariant, augmentation-preserving chain map
\[
 C_*^{\cell}(\partial_{\mathrm{top}}I^K;\F)
 \longrightarrow C_*^{\mathrm{poly}}(S(W_K);\F).
\]

The source is the cellular chain complex of
\[
 \partial_{\mathrm{top}}I^K\cong S^{K-1}=S^{d+1}.
\]
Its standard augmentation satisfies
\[
 \operatorname{im}\epsilon_{\partial_{\mathrm{top}}I^K}=\F,
 \qquad
 \ker\epsilon_{\partial_{\mathrm{top}}I^K}
 =\operatorname{im}\partial_1^{\cell},
\]
and cellular homology gives
\[
 \ker\partial_i^{\cell}
 =\operatorname{im}\partial_{i+1}^{\cell}
 \qquad(1\leq i\leq d).
\]
By Proposition~\ref{prop:polyhedral-augmentation}, the target is an
augmented chain complex and the antipodal map induces a chain
involution preserving its augmentation.  It has no chains above
degree \(d\), and
Lemma~\ref{lem:polyhedral-antipodal-exactness} gives
\[
 \ker\bigl(\id+(A_{W_K})_{\#,j}\bigr)
 =\operatorname{im}\bigl(\id+(A_{W_K})_{\#,j}\bigr)
 \qquad(0\leq j\leq d).
\]
When \(K=2\), one has \(d=0\), so the positive-degree exactness
conditions on the source are empty.  Thus, in every case, all
hypotheses of Lemma~\ref{lem:dold-chain} are satisfied.
That lemma forbids the displayed chain map, a contradiction.
\end{proof}

\section*{Acknowledgements}

This work is supported by the National Natural Science Foundation of China
under Grants 12371343 and 12525110.

The first author also thanks the Shanghai Institute for Mathematics and
Interdisciplinary Sciences (SIMIS), China, for its financial support. This research
was partly funded by SIMIS under Grant No.~SIMIS-ID-2024-WE. The first author is
grateful for the resources and facilities provided by SIMIS, which were essential
for the completion of this work.

\vspace{1em}
{\bfseries Statement of AI Use.} 
The central proof idea in this manuscript was first generated with the assistance of GPT 5.6 Sol Ultra. The manuscript was drafted with the assistance of Codex (using GPT 5.6 Sol), and subsequently revised and approved by the authors.


\begin{thebibliography}{99}

\bibitem{Bjorner}
A.~Bj\"orner,
Topological methods,
in \emph{Handbook of Combinatorics}, Vols.~1 and~2,
R.~L.~Graham, M.~Gr\"otschel, and L.~Lov\'asz, eds.,
Elsevier, Amsterdam, 1995, 1819--1872.

\bibitem{Borsuk}
K.~Borsuk,
Drei S\"atze \"uber die \(n\)-dimensionale euklidische Sph\"are,
\emph{Fundamenta Mathematicae} \textbf{20} (1933), 177--190.
\href{https://doi.org/10.4064/fm-20-1-177-190}
{doi:10.4064/fm-20-1-177-190}.

\bibitem{Dold}
A.~Dold,
Simple proofs of some Borsuk--Ulam results,
in \emph{Proceedings of the Northwestern Homotopy Theory Conference
(Evanston, Illinois, 1982)},
Contemporary Mathematics \textbf{19},
American Mathematical Society, 1983, 65--69.
\href{https://doi.org/10.1090/conm/019/711043}
{doi:10.1090/conm/019/711043}.

\bibitem{Dvorak}
V.~Dvo\v{r}\'ak,
A note on Norine's antipodal-colouring conjecture,
\emph{Electronic Journal of Combinatorics} \textbf{27} (2020),
Paper No.~P2.26.
\href{https://doi.org/10.37236/9219}{doi:10.37236/9219}.

\bibitem{Dzavoronok}
A.~D\v{z}avoronok,
Monochromatic paths and a topological approach to Norine's conjecture,
arXiv:2606.04181v2, 2026.
\url{https://arxiv.org/abs/2606.04181}.

\bibitem{EllisEtAl}
D.~Ellis, M.-R.~Ivan, I.~Leader, and J.~M.~Mackay,
Antipodal paths in covers of spheres,
arXiv:2607.13964v1, 2026.
\url{https://arxiv.org/abs/2607.13964}.

\bibitem{FederSubi}
T.~Feder and C.~S.~Subi,
On hypercube labellings and antipodal monochromatic paths,
\emph{Discrete Applied Mathematics} \textbf{161} (2013),
no.~10--11, 1421--1426.
\href{https://doi.org/10.1016/j.dam.2012.12.025}
{doi:10.1016/j.dam.2012.12.025}.

\bibitem{FrankstonScheinerman}
K.~Frankston and D.~Scheinerman,
Proving Norine's conjecture holds for \(n=7\) via SAT solvers,
arXiv:2408.02474v1, 2024.
\url{https://arxiv.org/abs/2408.02474}.

\bibitem{Hatcher}
A.~Hatcher,
\emph{Algebraic Topology},
Cambridge University Press, 2002.
Electronic edition:
\url{https://pi.math.cornell.edu/~hatcher/AT/AT.pdf}.

\bibitem{Hollom}
L.~Hollom,
Hypercube geodesics with few colour changes,
arXiv:2605.20184v3, 2026.
\url{https://arxiv.org/abs/2605.20184}.

\bibitem{KirchwegerEtAl}
M.~Kirchweger, T.~Peitl, B.~Subercaseaux, and S.~Szeider,
From the finite to the infinite: sharper asymptotic bounds on Norin's
conjecture via SAT,
arXiv:2511.08386v1, 2025.
\url{https://arxiv.org/abs/2511.08386}.

\bibitem{LeaderLong}
I.~Leader and E.~Long,
Long geodesics in subgraphs of the cube,
\emph{Discrete Mathematics} \textbf{326} (2014), 29--33.
\href{https://doi.org/10.1016/j.disc.2014.02.013}
{doi:10.1016/j.disc.2014.02.013}.

\bibitem{Lovasz}
L.~Lov\'asz,
Kneser's conjecture, chromatic number, and homotopy,
\emph{Journal of Combinatorial Theory, Series A} \textbf{25} (1978),
no.~3, 319--324.
\href{https://doi.org/10.1016/0097-3165(78)90022-5}
{doi:10.1016/0097-3165(78)90022-5}.

\bibitem{Matousek}
J.~Matou\v{s}ek,
\emph{Using the Borsuk--Ulam Theorem:
Lectures on Topological Methods in Combinatorics and Geometry},
Universitext,
Springer-Verlag, Berlin, 2003.
\href{https://doi.org/10.1007/978-3-540-76649-0}
{doi:10.1007/978-3-540-76649-0}.

\bibitem{Norine}
S.~Norine,
Edge-antipodal colorings of cubes,
\emph{Open Problem Garden}, posted by M.~DeVos, October 6, 2008.
\url{https://www.openproblemgarden.org/op/edge_antipodal_colorings_of_cubes}.

\bibitem{Soltesz}
D.~Solt\'esz,
On the 1-switch conjecture,
\emph{Discrete Mathematics} \textbf{340} (2017),
no.~7, 1749--1756.
\href{https://doi.org/10.1016/j.disc.2016.11.006}
{doi:10.1016/j.disc.2016.11.006}.

\bibitem{WestWise}
D.~B.~West and J.~I.~Wise,
Antipodal edge-colorings of hypercubes,
\emph{Discussiones Mathematicae Graph Theory} \textbf{39} (2019),
no.~1, 271--284.
\href{https://doi.org/10.7151/dmgt.2055}{doi:10.7151/dmgt.2055}.

\bibitem{Ziegler}
G.~M.~Ziegler,
\emph{Lectures on Polytopes},
Graduate Texts in Mathematics, vol.~152,
Springer, 1995.
\href{https://doi.org/10.1007/978-1-4613-8431-1}
{doi:10.1007/978-1-4613-8431-1}.

\bibitem{ZulkoskiEtAl}
E.~Zulkoski, C.~Bright, A.~Heinle, I.~S.~Kotsireas,
K.~Czarnecki, and V.~Ganesh,
Combining SAT solvers with computer algebra systems to verify
combinatorial conjectures,
\emph{Journal of Automated Reasoning} \textbf{58} (2017),
no.~3, 313--339.
\href{https://doi.org/10.1007/s10817-016-9396-y}
{doi:10.1007/s10817-016-9396-y}.

\end{thebibliography}
\end{document}